\documentclass{amsart}

\usepackage{amsfonts,amssymb,verbatim,amsmath,amsthm,latexsym,textcomp,amscd, hyperref}
\usepackage{epsfig}
\usepackage{amsmath,amsthm,graphics,float}
\usepackage{graphicx}
\usepackage{color}
\usepackage{url}
\usepackage{thm-restate}
\usepackage{MnSymbol}

\begin{document}

	\newtheorem{theorem}{Theorem}[section]
	\newtheorem{prop}[theorem]{Proposition}
	\newtheorem{lemma}[theorem]{Lemma}
	\newtheorem{cor}[theorem]{Corollary}
	\newtheorem{cond}[theorem]{Condition}
	\newtheorem{ing}[theorem]{Ingredients}
	\newtheorem{conj}[theorem]{Conjecture}
	\newtheorem{claim}[theorem]{Claim}
	\newtheorem{constr}[theorem]{Construction}
	\newtheorem{rem}[theorem]{Remark}
	
	\newtheorem*{theorem*}{Theorem}
	\newtheorem*{modf}{Modification for arbitrary $n$}
	\newtheorem{qn}[theorem]{Question}
	\newtheorem{condn}[theorem]{Condition}
	\newtheorem*{BGIT}{Bounded Geodesic Image Theorem}
	\newtheorem*{BI}{Behrstock Inequality}
	\newtheorem*{QCH}{Wise's Quasiconvex Hierarchy Theorem}
	
	\theoremstyle{definition}
	\newtheorem{defn}[theorem]{Definition}
	\newtheorem{eg}[theorem]{Example}
	\newtheorem{rmk}[theorem]{Remark}
	\newtheorem{remark}[theorem]{Remark}
	
	\newcommand{\map}{\rightarrow}
	\newcommand{\boundary}{\partial}
	\newcommand{\C}{{\mathbb C}}
	\newcommand{\integers}{{\mathbb Z}}
	\newcommand{\natls}{{\mathbb N}}
	\newcommand{\ratls}{{\mathbb Q}}
	\newcommand{\reals}{{\mathbb R}}
	\newcommand{\proj}{{\mathbb P}}
	\newcommand{\lhp}{{\mathbb L}}
	\newcommand{\tr}{{\operatorname{Tread}}}
	\newcommand{\rs}{{\operatorname{Riser}}}
	\newcommand{\tube}{{\mathbb T}}
	\newcommand{\cusp}{{\mathbb P}}
	\newcommand\AAA{{\mathcal A}}
	\newcommand\BB{{\mathcal B}}
	\newcommand\CC{{\mathcal C}}
	\newcommand\ccd{{{\mathcal C}_\Delta}}
	\newcommand\DD{{\mathcal D}}
	\newcommand\EE{{\mathcal E}}
	\newcommand\FF{{\mathcal F}}
	\newcommand\GG{{\mathcal G}}
	\newcommand\HH{{\mathcal H}}
	\newcommand\II{{\mathcal I}}
	\newcommand\JJ{{\mathcal J}}
	\newcommand\KK{{\mathcal K}}
	\newcommand\LL{{\mathcal L}}
	\newcommand\MM{{\mathcal M}}
	\newcommand\NN{{\mathcal N}}
	\newcommand\OO{{\mathcal O}}
	\newcommand\PP{{\mathcal P}}
	\newcommand\QQ{{\mathcal Q}}
	\newcommand\RR{{\mathcal R}}
	\newcommand\SSS{{\mathcal S}}
	
	\newcommand\TT{{\mathcal T}}
	\newcommand\ttt{{\mathcal T}_T}
	\newcommand\tT{{\widetilde T}}
	\newcommand\UU{{\mathcal U}}
	\newcommand\VV{{\mathcal V}}
	\newcommand\WW{{\mathcal W}}
	\newcommand\XX{{\mathcal X}}
	\newcommand\YY{{\mathcal Y}}
	\newcommand\ZZ{{\mathcal Z}}
	\newcommand\CH{{\CC\HH}}
	\newcommand\TC{{\TT\CC}}
	\newcommand\EXH{{ \EE (X, \HH )}}
	\newcommand\GXH{{ \GG (X, \HH )}}
	\newcommand\GYH{{ \GG (Y, \HH )}}
	\newcommand\PEX{{\PP\EE  (X, \HH , \GG , \LL )}}
	\newcommand\MF{{\MM\FF}}
	\newcommand\PMF{{\PP\kern-2pt\MM\FF}}
	\newcommand\ML{{\MM\LL}}
	\newcommand\mr{{\RR_\MM}}
	\newcommand\tmr{{\til{\RR_\MM}}}
	\newcommand\PML{{\PP\kern-2pt\MM\LL}}
	\newcommand\GL{{\GG\LL}}
	\newcommand\Pol{{\mathcal P}}
	\newcommand\half{{\textstyle{\frac12}}}
	\newcommand\Half{{\frac12}}
	\newcommand\Mod{\operatorname{Mod}}
	\newcommand\Area{\operatorname{Area}}
	\newcommand\ep{\epsilon}
	\newcommand\hhat{\widehat}
	\newcommand\Proj{{\mathbf P}}
	\newcommand\U{{\mathbf U}}
	\newcommand\Hyp{{\mathbf H}}
	\newcommand\D{{\mathbf D}}
	\newcommand\Z{{\mathbb Z}}
	\newcommand\R{{\mathbb R}}
	\newcommand\bN{\mathbb{N}}
	\newcommand\s{{\Sigma}}
	\renewcommand\P{{\mathbb P}}
	\newcommand\Q{{\mathbb Q}}
	\newcommand\E{{\mathbb E}}
	\newcommand\til{\widetilde}
	\newcommand\length{\operatorname{length}}
	\newcommand\BU{\operatorname{BU}}
	\newcommand\gesim{\succ}
	\newcommand\lesim{\prec}
	\newcommand\simle{\lesim}
	\newcommand\simge{\gesim}
	\newcommand{\simmult}{\asymp}
	\newcommand{\simadd}{\mathrel{\overset{\text{\tiny $+$}}{\sim}}}
	\newcommand{\ssm}{\setminus}
	\newcommand{\diam}{\operatorname{diam}}
	\newcommand{\pair}[1]{\langle #1\rangle}
	\newcommand{\T}{{\mathbf T}}
	\newcommand{\inj}{\operatorname{inj}}
	\newcommand{\pleat}{\operatorname{\mathbf{pleat}}}
	\newcommand{\short}{\operatorname{\mathbf{short}}}
	\newcommand{\vertices}{\operatorname{vert}}
	\newcommand{\collar}{\operatorname{\mathbf{collar}}}
	\newcommand{\bcollar}{\operatorname{\overline{\mathbf{collar}}}}
	\newcommand{\I}{{\mathbf I}}
	\newcommand{\tprec}{\prec_t}
	\newcommand{\fprec}{\prec_f}
	\newcommand{\bprec}{\prec_b}
	\newcommand{\pprec}{\prec_p}
	\newcommand{\ppreceq}{\preceq_p}
	\newcommand{\sprec}{\prec_s}
	\newcommand{\cpreceq}{\preceq_c}
	\newcommand{\cprec}{\prec_c}
	\newcommand{\topprec}{\prec_{\rm top}}
	\newcommand{\Topprec}{\prec_{\rm TOP}}
	\newcommand{\fsub}{\mathrel{\scriptstyle\searrow}}
	\newcommand{\bsub}{\mathrel{\scriptstyle\swarrow}}
	\newcommand{\fsubd}{\mathrel{{\scriptstyle\searrow}\kern-1ex^d\kern0.5ex}}
	\newcommand{\bsubd}{\mathrel{{\scriptstyle\swarrow}\kern-1.6ex^d\kern0.8ex}}
	\newcommand{\fsubeq}{\mathrel{\raise-.7ex\hbox{$\overset{\searrow}{=}$}}}
	\newcommand{\bsubeq}{\mathrel{\raise-.7ex\hbox{$\overset{\swarrow}{=}$}}}
	\newcommand{\tw}{\operatorname{tw}}
	\newcommand{\base}{\operatorname{base}}
	\newcommand{\trans}{\operatorname{trans}}
	\newcommand{\rest}{|_}
	\newcommand{\bbar}{\overline}
	\newcommand{\UML}{\operatorname{\UU\MM\LL}}
	\renewcommand{\d}{\operatorname{diam}}
	\newcommand{\hs}{{\operatorname{hs}}}
	\newcommand{\EL}{\mathcal{EL}}
	\newcommand{\tsh}[1]{\left\{\kern-.9ex\left\{#1\right\}\kern-.9ex\right\}}
	\newcommand{\Tsh}[2]{\tsh{#2}_{#1}}
	\newcommand{\qeq}{\mathrel{\approx}}
	\newcommand{\Qeq}[1]{\mathrel{\approx_{#1}}}
	\newcommand{\qle}{\lesssim}
	\newcommand{\Qle}[1]{\mathrel{\lesssim_{#1}}}
	\newcommand{\simp}{\operatorname{simp}}
	\newcommand{\vsucc}{\operatorname{succ}}
	\newcommand{\vpred}{\operatorname{pred}}
	\newcommand\fhalf[1]{\overrightarrow {#1}}
	\newcommand\bhalf[1]{\overleftarrow {#1}}
	\newcommand\sleft{_{\text{left}}}
	\newcommand\sright{_{\text{right}}}
	\newcommand\sbtop{_{\text{top}}}
	\newcommand\sbot{_{\text{bot}}}
	\newcommand\sll{_{\mathbf l}}
	\newcommand\srr{_{\mathbf r}}
	\newcommand\geod{\operatorname{\mathbf g}}
	\newcommand\mtorus[1]{\boundary U(#1)}
	\newcommand\A{\mathbf A}
	\newcommand\Aleft[1]{\A\sleft(#1)}
	\newcommand\Aright[1]{\A\sright(#1)}
	\newcommand\Atop[1]{\A\sbtop(#1)}
	\newcommand\Abot[1]{\A\sbot(#1)}
	\newcommand\boundvert{{\boundary_{||}}}
	\newcommand\storus[1]{U(#1)}
	\newcommand\Momega{\omega_M}
	\newcommand\nomega{\omega_\nu}
	\newcommand\twist{\operatorname{tw}}
	\newcommand\SSSS{{\til{\mathcal S}}}
	\newcommand\modl{M_\nu}
	\newcommand\MT{{\mathbb T}}
	\newcommand\dw{{d_{weld}}}
	\newcommand\dt{{d_{te}}}
	\newcommand\Teich{{\operatorname{Teich}}}
	\renewcommand{\Re}{\operatorname{Re}}
	\renewcommand{\Im}{\operatorname{Im}}
	\newcommand{\mc}{\mathcal}
	\newcommand{\ccs}{{\CC(S)}}
	\newcommand{\mtdw}{{(\til{M_T},\dw)}}
	\newcommand{\tmtdw}{{(\til{M_T},\dw)}}
	\newcommand{\tmldw}{{(\til{M_l},\dw)}}
	\newcommand{\mtdt}{{(\til{M_T},\dt)}}
	\newcommand{\tmtdt}{{(\til{M_T},\dt)}}
	\newcommand{\tmldt}{{(\til{M_l},\dt)}}
	\newcommand{\trvw}{{\tr_{vw}}}
	\newcommand{\ttrvw}{{\til{\tr_{vw}}}}
	\newcommand{\but}{{\BU(T)}}
	\newcommand{\ilkv}{{i(lk(v))}}
	\newcommand{\pslc}{{\mathrm{PSL}_2 (\mathbb{C})}}
	\newcommand{\tttt}{{\til{\ttt}}}
	\newcommand{\bcomment}[1]{\textcolor{blue}{#1}}
	\newcommand{\jfm}[1]{\marginpar{#1\quad -jfm}}
	
	\newcommand{\defstyle}[1]{\textbf{#1}}
	\newcommand{\emphstyle}[1]{\emph{#1}}
	
	\title{Property (T) for fiber products}
	
	\author{Mahan Mj}
	\address{School of Mathematics, Tata Institute of Fundamental Research, 1 Homi Bhabha Road, Mumbai 400005, India}
	
	\email{mahan@math.tifr.res.in}
	\email{mahan.mj@gmail.com}
	\urladdr{http://www.math.tifr.res.in/~mahan}
	
	\author{Arghya Mondal}
	\address{Chennai Mathematical Institute, H1, SIPCOT IT Park, Siruseri, Kelambakkam 603103,
		India}
	\email{amondal@cmi.ac.in}
	\email{mondalarghya1990@gmail.com}
	\urladdr{https://sites.google.com/view/arghyamondal/home}

	\thanks{Both authors are  supported by  the Department of Atomic Energy, Government of India, under project no.12-R\&D-TFR-14001.
		MM is supported in part by a Department of Science and Technology JC Bose Fellowship,  and an endowment of the Infosys Foundation via the Chandrasekharan-Infosys Virtual Centre for Random Geometry.}

	\date{\today}
	
	

	\def\f{\mathfrak}
	\def\r{\mathbb{R}}
	\def\c{\mathbb{C}}
	\def\h{\mathbb{H}}
	\def\n{\mathbb{N}}
	\def\q{\mathbb{Q}}
	\def\z{\mathbb{Z}}
	\def\G{\Gamma}
	\def\g{\gamma}
	\def\b{\backslash}
	
	\begin{abstract}
		We study when the fiber product of groups with Property (T) has Property (T). 
	\end{abstract}

	\maketitle

	\section{Introduction}
	It is well-known and easy to see that the product of two groups $G_1 \times G_2$ has Kazhdan's property (T) \cite{BdlHV} if and only if both $G_1$ and $G_2$ have Property (T). The aim of this article is to study the following question:
	
	\begin{qn}\label{mainqn}
		Let $G_1$ and $G_2$ have Property (T). Let $q_1: G_1 \to H$ and $q_2: G_2 \to H$ be homomorphisms. When does $G_1 \times_H G_2$ 
		have Property (T)?
	\end{qn}
	
	For simplicity, we will assume all our groups to be discrete in this introductory section, though in some cases more general versions of the results are proved in the main body of the paper. 
	
	It turns out that Question \ref{mainqn} does not have a uniform answer, and we provide both  sufficient condition and counter-examples. We start (Proposition \ref{prop-split}) by noting the sufficient condition  that  one of the maps $q_1$ or $q_2$ admits a section. This allows us to describe $G_1\times_HG_2$ in terms of a short exact sequence and  apply the criterion for extension of a Property (T) group to have Property (T). 
	
	We turn to counterexamples now. Historically the first examples of groups with Property (T) were lattices in higher rank semi-simple algebraic groups. If $G_1$ is such a group then the kernel of the homomorphism $q_1$ (in Question \ref{mainqn}) is a normal subgroup of $G_1$ and hence, by Margulis Normal Subgroup theorem, is either finite or of finite index in $G_1$. Then it is clear that $G_1\times_HG_2$ has Property (T). This and other counterexamples in this paper depend on deep results by other authors deriving Property (T) for groups having other interesting properties. For instance, our first counterexample (Proposition \ref{prop-central}) involves groups with Property (T) whose center is non-compact. Examples of such groups are based on Serre's result that if $G/Z(G)$ has Property (T) and $G/[G,G]$ is finite then $G$ has Property (T). The second counterexample (Proposition \ref{infgen}) involves groups with Property (T) that are not finitely presented. Examples of such groups were given by Margulis, Gromov, and Cornulier. We also make essential use of Shalom's result that any infinitely presented group with Property (T) is the quotient of a finitely presented group with Property (T). The third counterexample (Theorem \ref{thm-hypemb}) involves groups with Property (T) that are hyperbolic, or more generally, admit hyperbolically embedded free or surface groups.  Once we assume that the groups $G_1, G_2$ or $H$ have a certain property, the failure of the fiber product to have Property (T) follows from  applying  known results about groups having such properties. The arguments do not really involve Property (T) deeply. We simply show that the fiber products we construct violate one of  two immediate consequences of Property (T): finite generation and finite abelianization. 
	
	Both the first and third counterexamples are fiber products of two copies of the same homomorphism $q:G\to H$. We have a good handle on this special case since $G\times_HG$ is isomorphic to the semi-direct product $N\rtimes G$ (Lemma \ref{lem-gpth}), where $N$ is the kernel of $q$. In this case Question \ref{mainqn} reduces to the following question:
	\begin{qn}\label{semidirect}
		When does a semi-direct product $N\rtimes G$ have Property (T), given that $G$ has Property (T)? 
	\end{qn}
    Equivalently, when does $(N\rtimes G,N)$ have relative Property (T) given that $G$ has Property (T) (cf. Lemma \ref{extension})? Relative property (T) has been studied (implicitly) since the inception of Property (T) by Kazhdan \cite{kazhdan}. Jolisaint \cite{jolissaint} gave equivalent formulations of relative Property (T), similar to those for Property (T). These were extended by Cornulier \cite{cornulier} to the case of pairs $(G,X)$, where $X$ is only a subset of $G$. Shalom \cite[Theorem 5.5]{shalom} gave a sufficeint condition for $(N\rtimes G,N)$ to have relative Property (T), where $N$ is abelian, in terms of invariant means on the unitary dual $\widehat{N}$ of $N$. Working with the same hypothesis, Cornulier and Tessera \cite{cortes} gave an equivalent formulation in terms of both means and probability measures on $\widehat{N}$. (See also \cite{burger, margulis, shalomIHES}.) Our results are in line with these developments. Let us recall some of the objects involved in characterizing Property (T) or its relative versions. We denote by $\mathcal{P}_1(N)^G$ (resp. $\mathcal{N}(N)^G$) the fixed points of the induced $G$ action on the set $\mathcal{P}_1(N)$ (resp. $\mathcal{N}(N)$) of normalized functions of positive type (resp. functions conditionally of negative type) on $N$. If $N$ is abelian we denote by $\mathcal{M}(\widehat{N})^G$ the fixed points of the induced $G$ action on the set $\mathcal{M}(\widehat{N})$ of regular probability measures on $\widehat{N}$. The Dirac measure at the trivial representation $1$ is denoted by $\delta_1$. All the spaces above are equipped with appropriate topologies (see \S\ref{sec-ns} for details). The following result is a combination of Theorem \ref{fupoty}, Theorem \ref{funety} and Theorem \ref{prob}.
	
	\begin{theorem}\label{allequiv}
		Let $G$ be a group having Property (T), which acts by automorphisms on a group $N$. The following are equivalent:
		\begin{enumerate}
			\item $N\rtimes G$ has Property (T)
			\item Every sequence $\{f_n\}\subset\mathcal{P}_1(N)^G$, that converges uniformly on compact subsets to the constant function $1$, converges uniformly on $N$. 
			\item Every function in $\mathcal{N}(N)^G$ is bounded.
		\end{enumerate}
		In addition, if $N$ is abelian then we also have
		\begin{enumerate}
			\item[(4)] There is no sequence $\{\mu_n\}\subset\mathcal{M}(\widehat{N})^G$ such that $\mu_n\to\delta_1$ and $\mu_n(1)=0$ for all $n\in\n$. 
		\end{enumerate}
	\end{theorem}
	Note that if $G$ is the trivial group, equivalence of the first three items is just the usual characterization of Property (T) (equivalently Property (FH)) in terms of functions of positive type or functions conditionally of negative type. Heuristically, Theorem \ref{allequiv} says that since $G$ already has Property (T), it should be possible to check whether $N\rtimes G$ has Property (T) purely in terms of unitary representations or affine isometric actions of $N$ and the $G$-action on $N$. The key point in the proof is that if a representation of $N\rtimes G$ has a $G$-invariant vector $v$ then the cyclic $N$-subspace generated by $v$ is invariant under the $G$-action. The converse of this observation (Lemma \ref{GNSfupoty}) allows us to extend any $G$-invariant cyclic representation of $N$ to a representation of $N\rtimes G$ on the same vector space. The analogous statement for affine isometric actions is Lemma \ref{GNSfunety} and in the abelian case is Lemma \ref{splrepns}. 
	
	Theorem \ref{allequiv} underscores the importance of the $G$ action on $N$ in addressing Question \ref{semidirect}. If it is known that $\text{Aut}(N)$ has Property (T) (or, in the opposite direction, Haagerup property), then  the image of $G$ in $\text{Aut}(N)$ helps in  deciding if  $N\rtimes G$ has Property (T). If $N$ is finitely generated abelian then, modulo a finite subgroup, the automorphism group of $N$ is $\text{GL}_n(\z)$, where $n$ is the rank of $N$. Then we observe, using a result of Raja \cite{raja}, that $N\rtimes G$ does not have Property (T) if $n\le 2$ and it has Property (T) if the image of $G$ is Zariski dense in $\text{SL}_n(\r)$ (Proposition \ref{fingenab}). In \cite{cms} Chatterji, Witte Morris and Shah prove a relative version of the result due to Serre mentioned before. Assume  $N$ is nilpotent. Then  \cite{cms} deduces relative Property (T) of a pair $(G,N)$, from that of $(G/[N,N],N/[N,N])$ via induction on the length of the lower central series of $N$.  As an immediate implication we note that our last result continues to hold if we replace $N$ by $N/[N,N]$, when $N$ is nilpotent (Proposition \ref{prop-nilp}). 
	
	Here is a section-wise outline of our paper. In \S\ref{sec-prelim} we collect all the known or elementary results related to Property (T) that we will need. In \S\ref{sec-ctreg} we note the sufficient condition Proposition \ref{prop-split} and the counterexamples mentioned above. In \S\ref{sec-ns} we prove all the statements of Theorem \ref{allequiv} in full generality. In \S\ref{sec-special} we discuss Question \ref{semidirect} when $N$ is finitely generated abelian or nilpotent.
	
	\section{Preliminaries}\label{sec-prelim}
	
	We  begin with the observation that the fiber product of two copies of the same map can be written as a semi-direct product. This observation is essentially due to Bass and Lubotzky  \cite[\S6]{bl}, and we include a proof here for completeness.
	%We start with the following basic group-theoretic Lemma.
	
	\begin{lemma}\label{lem-gpth}
		Consider two copies of the same surjective homomorphism $f:G\to H$ with kernel $N$. Then $G\times_HG\cong N\rtimes G$.
	\end{lemma}
	
	\begin{proof}
		It is enough to identify a normal subgroup $N'$ of $G\times_HG$ that is isomorphic to $N$ and a subgroup $G'$ which is isomorphic to $G$, such that $N'\cap G'$ is trivial and $G\times_HG=N'G'$. Take $N':=\{1\}\times N$ and $G':=\{(g_1,g_2)\in G\times_HG:g_1=g_2\}$. Clearly $N'\cap G'=\{(1,1)\}$. Let $(g_1,g_2)\in G\times_HG$. Then $f(g_1)=f(g_2)$. So there exists $n\in N$ such that $ng_1=g_2$. Hence $(g_1,g_2)=(1,n)(g_1,g_1)$. Thus $G\times_HG=N'G'$.
	\end{proof}

	%*******************************************
	
	\subsection{GNS correspondences} In this subsection we will recall the correspondences between unitary representations/affine actions of a group and some function/measure spaces. These results will be needed in \S\ref{sec-ns}. Recall that functions of positive type on a topological group $G$ which take the value 1 at the identity $e$ are called \emph{normalized} and the set of normalized functions of positive type on $G$ is denoted by $\mathcal{P}_1(G)$.
	The first GNS correspondence is the following.
	
	\begin{theorem}\label{orgGNSfupoty}\cite[Theorem C.4.10]{BdlHV}
		There is a one to one correspondence
		\begin{center}
			\begin{tabular}{ccc}
				$\{(\pi,V_\pi,v):(\pi,V_\pi)$ is a unitary  &  &  \\
				cyclic representation of $G$ and & $\leftrightarrow$  & $\mathcal{P}_1(G)$, \\ $v\in V_\pi$ a cyclic unit vector$\}/\sim$ & &
			\end{tabular}
		\end{center}
		where $(\pi,V_\pi,v)\sim (\rho,V_\rho,w)$ if there exists an isometric intertwining operator $T:V_\pi\to V_\rho$ such that $T(v)=w$. The correspondence is given by $(\pi,V_\pi,v)\mapsto (x\mapsto \langle\pi(x)v,v\rangle)$.
	\end{theorem}
	
	The second correspondence is between affine isometric actions of $G$ and functions conditionally of negative type on $G$. In \cite{BdlHV} this GNS correspondence is not stated in the form below, but it can be easily derived from the relevant results there. We will supply the extra detail. First, in analogy to cyclic representations, we define a \emph{cyclic affine action} to be an affine action $(\alpha,\mathcal{H})$ for which the span of $\{\alpha(x)(0):x\in G\}$ is dense in $\mathcal{H}$. Functions conditionally of negative type on $G$ which take the value $0$ at the identity $e$ will be called \emph{normalized}. Let the space of such functions be denoted by $\mathcal{N}_0(G)$.
	
	\begin{theorem}\label{orgGNSfunety}\cite[Proposition 2.10.2]{BdlHV}
		There is a one to one correspondence
		\begin{center}
			\begin{tabular}{ccc}
				$\{(\alpha,\mathcal{H}):(\alpha,\mathcal{H})$ is a cyclic affine isometric action$\}/\sim$ & $\leftrightarrow$ & $\mathcal{N}_0(G)$, \\
				%$ & &
			\end{tabular}
		\end{center}
		where $(\alpha,\mathcal{H})\sim (\alpha_1,\mathcal{H}_1)$ if there exists a $G$-equivariant orthogonal map $T:\mathcal{H}\to \mathcal{H}_1$. The correspondence is given by $(\alpha,\mathcal{H})\mapsto (x\mapsto\|\alpha(x)(0)\|^2)$.
	\end{theorem}
	
	\noindent\textit{Proof of well-definedness of the correspondence:} Let $(\alpha,\mathcal{H})$ and $(\alpha_1,\mathcal{H}_1)$ be two cyclic affine isometric actions of $G$ such that $\|\alpha(x)(0)\|^2=\|\alpha_1(x)(0)\|^2=:\psi(x)$, for all $x\in G$. Consider the kernel conditionally of negative type $\Psi:G\times G\to\r$ associated to $\psi$, that is, $\Psi(x,y)=\psi(y^{-1}x)$. Using the fact that $\alpha$ is a group homomorphism and that each element in its image is an affine action, we get  
	\begin{equation*}
		\Psi(x,y)=\|\alpha(y^{-1}x)(0)\|^2=\|\alpha(y^{-1}x)(0)-\alpha(e)(0)\|^2=\|\alpha(x)(0)-\alpha(y)(0)\|^2.
	\end{equation*}
	Similarly $\Psi(x,y)=\|\alpha_1(x)(0)-\alpha_1(y)(0)\|^2$. We are also given that the spans of $\{\alpha(x)(0)-\alpha(e)(0):x\in G\}$ and $\{\alpha_1(x)(0)-\alpha_1(e)(0):x\in G\}$ are dense in $\mathcal{H}$. Uniqueness of GNS construction for kernels conditionally of negative type \cite[Theorem C.2.3]{BdlHV} tells us that there exists a unique affine isometry $A:\mathcal{H}\to\mathcal{H}_1$, such that $\alpha_1(x)(0)=A(\alpha(x)(0))$ for all $x\in G$. In particular $A(0)=A(\alpha(e)(0))=\alpha_1(e)(0)=0$. An affine map that takes $0$ to $0$ is an orthogonal operator. Hence $A$ is an orthogonal operator. Also clearly $A$ is $G$-equivariant on the span of $\{\alpha(x)(0):x\in G\}$, and the latter is dense in $\mathcal{H}$. Hence by continuity $A$ is $G$-equivariant on $\mathcal{H}$. \qed
	
	Finally we have that if $G$ is abelian then  cyclic representations also correspond to regular probability measures on the unitary dual $\widehat{G}$. This is well known to experts, but  we supply a proof for completeness.
	
	\begin{lemma}\label{directint}
		Let $G$ be a locally compact abelian group. Let $\widehat{G}$ be its unitary dual and let $\mathcal{M}(\widehat{G})$ be the space of regular probability measures on $\widehat{G}$. There is a one to one correspondence
		\begin{center}
			\begin{tabular}{ccc}
				$\{(\pi,V_\pi,v):(\pi,V_\pi)$ is a unitary  & & \\
				cyclic representation of $G$ and & $\leftrightarrow$ & $\mathcal{M}(\widehat{G})$, \\
				$v\in V_\pi$ a cyclic unit vector$\}/\sim$ & &
			\end{tabular}
		\end{center}
		where $(\pi,V_\pi,v)\sim (\rho,V_\rho,w)$ if there exists an isometric intertwining operator $T:V_\pi\to V_\rho$ such that $T(v)=w$. 
	\end{lemma}

	\begin{proof}
		Let us describe the map which gives the correspondence. First note that any regular measure $\mu$ on $\widehat{N}$ gives us a representation $(\rho,L^2_\mu(\widehat{N}))$ of $N$, where $\rho(x)(f)(\chi):=\chi(x)f(\chi)$. We claim that the constant function $1$ is a cyclic vector. It is enough to check that $\rho(L^1(N))1$ is dense in $L^2_\mu(\widehat{N})$. But $\rho(L^1(N))1$ is the image of $L^1(N)$ under the Fourier transform and hence is dense in $C_0(\widehat{N})$, which is itself dense in $L^2_\mu(\widehat{N})$, since $\mu$ is finite regular. The direct integral decomposition of unitary representations of locally compact abelian groups \cite[Theorem 7.28]{folland} implies that any equivalence class in the left hand side of the above correspondence contains a pair $(\rho,L^2_\mu(\widehat{N}),1)$, for some regular probability measure $\mu$. This measure must be unique, since if there is a intertwining map $L^2_{\mu_1}(\widehat{N})\to L^2_{\mu_2}(\widehat{N})$ sending $1$ to $1$,  it must be the identity on $C_0(\widehat{N})$. Hence, $\mu_1=\mu_2$ by the Riesz Representation Theorem. The correspondence is given by the map which sends the equivalence class of $(\pi,V_\pi,v)$ to $\mu$, where $(\rho,L^2_\mu(\widehat{N}),1)$ is the unique such representation in that class. This map is a bijection by the above discussion. 
	\end{proof}
	\subsection{Results about Property (T)} All the results here, except the first one, will be required in \S\ref{sec-ns}. The first result is about when an extension of a group with Property (T) has Property (T).
	\begin{lemma}\label{extension}\cite[Remark 1.7.7]{BdlHV}
		Given a short exact sequence of groups
		\begin{equation*}
			1\to N\to G\to Q\to 1,
		\end{equation*}
		$G$ has Property (T) if and only if $Q$ has Property (T) and $(G,N)$ has relative Property (T). 
	\end{lemma}
	The next result states that in a group having Property (T), one can choose almost invariant vectors as close to invariant vectors as one wishes. 
	
	\begin{prop}\label{alinv}\cite[Chapitre 1, Proposition 16]{dlHV}
		Let $K$ be a compact generating set of a group $G$ having Property (T). Given any $0<\delta\le 2$, there exists $\epsilon>0$, such that in any unitary representation $(\pi,V_\pi)$ of $G$, if $v$ is a $(K,\epsilon)$-invariant unit vector, then there exists an invariant unit vector $w\in V_\pi$, such that $\|v-w\|<\delta$.
	\end{prop}
	
	The next two results are a consequence of  \cite[Lemma 2.2.7]{BdlHV} which says that among all closed balls containing a non-empty bounded subset of a real or complex Hilbert space, there exists a unique one with minimal radius. The center of this unique minimal closed ball is the \emph{center of $X$}. The significance of this lemma for isometric group actions is that if $X$ is a $G$-invariant bounded subset of a Hilbert space then the center of $X$ is a $G$-fixed point. When the action is via a unitary representation we have the following consequence.
	
	\begin{lemma}\label{bddorbitur}\cite[Chapitre 3, Corollaire 11]{dlHV}
		Let $(\pi,V_\pi)$ be a unitary representation of a group $G$. Let $v\in V_\pi$ be such that $\emph{Re}\langle\pi(g)v,v\rangle\ge\epsilon$, for some $\epsilon>0$ and all $g\in G$. Then $\pi(G)$ has a non-zero invariant vector.   
	\end{lemma}
	
	When the action is via affine isometries on a real Hilbert space we have the following consequence.
	
	\begin{lemma}\label{bddorbitai}\cite[Proposition 2.2.9]{BdlHV}
		Let $(\alpha,\mathcal{H})$ be an affine isometric action of a group $G$. Let $b$ be the corresponding cocycle. Then $\alpha$ has a fixed vector if and only if $b$ is bounded. 
	\end{lemma}

	\section{A Sufficient condition and counterexamples}\label{sec-ctreg}
	We begin with a simple sufficient condition.
	\begin{prop}\label{prop-split}
		Let $G_1$ and $G_2$ be two groups having Property (T). If either of the surjective homomorphisms $f_1:G_1\to H$ and $f_2:G_2\to H$ splits, then $G_1\times_HG_2$ also has Property (T).
	\end{prop}
	
	\begin{proof}
		Assume, without loss of generality, that $f_1$ splits. The idea is to apply Lemma \ref{extension} to the short exact sequence
		\begin{equation*}
			1\to\{1\}\times N_2\to G_1\times_HG_2\to G_1\to 1.
		\end{equation*}
		Thus all we need to show is that $(G_1\times_HG_2,\{1\}\times N_2)$ has relative Property (T). If we have a sequence of subgroups $N<H<G$ with $H$ having Property (T), then $(G,N)$ has Property (T). Now we will use the splitting of $f_1$ to embed a copy of $G_2$ in $G_1\times_H G_2$ containing $\{1\}\times N_2$. This will finish the proof. Let $\phi$ be the splitting map. Then consider the subgroup $G_2':=\{(\phi(h),g):h\in H,f_2(g)=h\}$  of $G_1\times_HG_2$. The restriction of the projection $G_1\times_HG_2\to G_2$ to $G_2'$ is an isomorphism and $\{1\}\times N_2\subset G'_2$. 
	\end{proof}
	
	The rest of this section is devoted to three classes of counterexamples. 
	
	\subsection{Counterexample: Non-compact center}
	Let $H$ be any simple Lie group with trivial center associated to one of the higher rank Hermitian symmetric spaces of non-compact type. Let $G$ be the universal cover of $H$. Then $G\to H$ is an infinite cover and $G$ has Property (T) due to a result of Serre  \cite[Remark 3.5.5 (iii)]{BdlHV}. Now it follows from the proposition below that $G\times_HG$ does not have Property (T). For an example with $G$ discrete, replace $H$ with a lattice $\G$ in it and $G$ with the pre-image $\G$ in $G$.

	\begin{prop}\label{prop-central}
		Let $G\to H$ be a surjective group homomorphism whose kernel $N$ is contained in the center Also assume that $N$ is not compact. Then $G\times_H G$ does not have Property (T).
	\end{prop}

	\begin{proof}
		By Lemma \ref{lem-gpth}, $G\times_HG\cong N\rtimes G$. Since $N$ is contained in the center of $G$, therefore the $G$-action on $N$ is trivial. Thus $G\times_HG\cong N\times G$, which cannot have Property (T) since its quotient $N$ being non-compact abelian, does not have Property (T).
	\end{proof}
	
	\subsection{Counterexample: Infinite presentation}
	\begin{prop}\label{infgen}
		For $i=1,2$, let $f_i:G_i\to H$, be  surjective homomorphisms from  finitely presented groups $G_i$, having Property (T), to a group $H$ which is not finitely presented. Then $G_1\times_H G_2$ does not have Property (T).
	\end{prop}
	\begin{proof}
		In fact $G_1\times_HG_2$ is not finitely generated. This is a special case of a result for subdirect products \cite[Corollary 2.4]{bm} by Bridson and Miller. 
	\end{proof}
	
	Examples of groups having Property (T) which are not finitely presented include
	\begin{enumerate}
		\item $H=\text{SL}_3(\mathbb{F}_p[X])$ (Margulis), 
		\item $H=\text{Sp}_4(\z[1/p])\ltimes\z[1/p]^4$, $p$ prime (Cornulier),
		\item $H$ is an infinite torsion quotient of uniform lattices in $\text{Sp}(n,1), n\ge 2$
		(Gromov).
	\end{enumerate}
	See \cite[\S3.4]{BdlHV} for more details.  Existence of finite presented groups $G_i$ with Property (T) which admit surjective homomorphisms to a group $H$ as above is guaranteed by a result of Shalom \cite[Theorem 3.4.5]{BdlHV} which asserts that any discrete group with Property (T) is the quotient of a finitely presented group with Property (T). 

	\subsection{Counterexample: Hyperbolically embedded free or surface groups}
	In this subsection all groups are assumed to be discrete. We start by recording the following observation. 
	
	\begin{lemma}\label{abel}
		Let $G$ be a group acting by automorphisms on another group $N$. Let $\phi$ be the action. Denote by $[G,N]$ the subgroup of $N$ generated by elements of the form $\phi(g)(n)n^{-1}$, where $g\in G$ and $n\in N$. Then the abelianization of $N\rtimes G$ is   $(N/[G,N])\times G_{ab}$. 
	\end{lemma}
		
	\begin{defn}\cite[Definition 2.2]{sun}
		Let $G$ be a group, $H$ be a subgroup and $N$ be a normal subgroup of $H$. We denote the normal closure of $N$ in $G$ by $\llangle N\rrangle$. The triple $(G,H,N)$ is said to have the \emph{Cohen-Lyndon property} if there exists a set $T$ of left coset representatives of $H\llangle N\rrangle$ in $G$ such that $\llangle N\rrangle$ is the free product of its subgroups $tNt^{-1}, t\in T$.
	\end{defn}
	
	\begin{prop}\label{prop-cl}
		Let $(G,H,N)$ have the Cohen-Lyndon property. Let $N/[H,N]$ be infinite. Then the fiber product of two copies of the quotient map $G\to G/\llangle N\rrangle$ has infinite abelianization, and hence does not have Property (T).
	\end{prop}
	
	\begin{proof}
		By Lemma \ref{abel}, it is enough to show that there exists a surjective homomorphism $\llangle N\rrangle/[G,\llangle N\rrangle]\to N/[H,N]$. To construct this homomorphism we will first define a homomorphism from each free factor of $\llangle N\rrangle$ to $N/[H,N]$. By the universal property of free products, this will define a unique homomorphism $\llangle N\rrangle\to N/[H,N]$. It will be surjective by construction. Finally we will show that $[G,\llangle N\rrangle]$ belongs to the kernel of this homomorphism. 
		
		For each $t\in T$, define the homomorphism $tNt^{-1}\to N/[H,N]$ to be a composition of the natural isomorphism, $tNt^{-1}\to N, m\mapsto t^{-1}mt$, and the quotient map $N\to N/[H,N]$. This defines a surjective homomorphism $\phi:\llangle N\rrangle\to N/[H,N]$. Note that, for any $n\in N$ and $t\in T$,
		\begin{equation*}
			\phi(tnt^{-1})=\phi(n).
		\end{equation*}
		This equation will be used repeatedly in calculations without comment. To show that $[G,\llangle N\rrangle]\subset\text{ker }\phi$, it is enough to show that each of the generators of $[G,\llangle N\rrangle]$ goes to $1$ under $\phi$. Any generator of $[G,\llangle N\rrangle]$ is of the form $gwg^{-1}w^{-1}$, where  $g\in G$ and $w\in\llangle N\rrangle$. Thus we have to show that for each $g\in G$ and $w\in\llangle N\rrangle$,
		\begin{equation}\label{ker}
			\phi(g)\phi(w)\phi(g)^{-1}=\phi(w).
		\end{equation}
		Let us first reduce (\ref{ker}) to the case where $w\in N$ and $g\in\llangle N\rrangle, H$ or $T$. Any $w\in\llangle N\rrangle$ is of the form $\prod_{i=1}^kt_in_it_i^{-1}$, where $n_i\in N$ and $t_i\in T$. Then
		\begin{align*}
			&\phi(g)\phi(w)\phi(g)^{-1}=\phi(g)(\prod_{i=1}^k\phi(t_in_it_i^{-1}))\phi(g)^{-1}\\
			=~&\phi(g)(\prod_{i=1}^k\phi(n_i))\phi(g)^{-1}=\prod_{i=1}^k\phi(g)\phi(n_i)\phi(g)^{-1}.
		\end{align*}
		Thus it is enough to show that (\ref{ker}) is true for $w\in N$. Since $T$ is a set of left coset representatives of $H\llangle N\rrangle$, therefore any $g\in G$ is of the form $g=thu$, where $t\in T, h\in H$ and $u\in\llangle N\rrangle$. Then $\phi(g)\phi(w)\phi(g)^{-1}=\phi(t)[\phi(h)[\phi(u)\phi(w)\phi(u)^{-1}]\phi(h)^{-1}]\phi(t)^{-1}$. Hence it is enough to prove (\ref{ker}) for $g$ belonging to $\llangle N\rrangle, H$ or $T$, separately. When $w\in N$ and $g\in T$ or $H$, (\ref{ker}) is immediate. So let us assume that $w\in N$ and $g\in\llangle N\rrangle$. The element $g$ is of the form $\prod_{i=1}^kt_in_it_i^{-1}$, where $n_i\in N$ and $t_i\in T$. By induction on the word length $k$ of $g$, we are reduced to the case $g=tmt^{-1}$, where $t\in T$ and $m\in N$. Then $\phi(g)\phi(w)\phi(g^{-1})=\phi(tmt^{-1})\phi(w)\phi(tmt^{-1})^{-1}=\phi(m)\phi(w)\phi(m)^{-1}$. But $mwm^{-1}w^{-1}\in [N,N]\subset [H,N]$. Hence $\phi(m)\phi(w)\phi(m)^{-1}=\phi(w)$.     
	\end{proof}
	
	\begin{remark}
		The homomorphism $\phi:\llangle N\rrangle/[G,\llangle N\rrangle]\to N/[H,N]$, in the above proof, is in fact an isomorphism. As a candidate for the inverse homomorphism, consider the map $\psi:N/[H,N]\to \llangle N\rrangle/[G,\llangle N\rrangle]$ induced by the natural inclusion $N\to \llangle N\rrangle$. Since $H\subset G$ and $N\subset \llangle N\rrangle$, therefore $[H,N]\subset[G,\llangle N\rrangle]$, and hence $\psi$ is well defined. That $\phi\psi=\text{id}$ is immediate. To see the other direction, we claim that the pre-composition of $\psi\phi$ by the quotient $q:\llangle N\rrangle\to\llangle N\rrangle/[G,\llangle N\rrangle]$ is equal to $q$. By the universal property of quotient maps this will imply that $\psi\phi=\text{id}$. The equality follows from comparing the two maps on each free factor.   
	\end{remark}
	
	To apply Proposition \ref{prop-cl}, we shall require the following consequence of Chevalley-Weil theory (in the case of free groups, the theorem is due to Gasch\"utz). See ~\cite{cw,gllm, kob}. In the form we reproduce this below, it is an immediate consequence of Theorems 3.3 and 3.4 of \cite{bkms}. Let $Q$ denote the finite  group $H/N$.
	The group $N/[N,N]$ can then be regarded as a 
	$Q$-module. Then $N/[H,N]$ is the space of co-invariants of the $Q$ action on $N/[N,N]$.
	Chevalley-Weil theory gives a non-trivial infinite summand of $N/[N,N]$  corresponding to the trivial representation. 
	It follows that $N/[H,N]$ is infinite. We summarize this below.
	
	\begin{theorem}~\cite{cw,gllm, kob},\cite[Theorems 3.3 and 3.4]{bkms}\label{cw} Let $H$ be either the free group $F_n$ \emph{(}$n>1$\emph{)} or
		the fundamental group of a 
		closed surface of genus greater than one. Let $N < H$
		be a proper finite index normal subgroup. Then $N/[H,N]$ is infinite.
	\end{theorem}
	
	\begin{remark}\label{rmk-ab}
		If $H$ is infinite abelian and $N < H$
		is a proper finite index (necessarily normal) subgroup,  then $N/[H,N]$ is equal to $N$ and is therefore infinite.
	\end{remark}
	
	We refer the reader to \cite{dgo} for the notion of {\bf hyperbolically embedded}
	subgroups. Examples
	include quasiconvex malnormal subgroups of hyperbolic groups and peripheral
	subgroups of (strongly) relatively hyperbolic groups.
	One  says \cite{dgo,sun} that a property $\PP$
	holds for {\bf all sufficiently deep} normal subgroups
	$N<H$ if there exists a finite set $\FF \subset H\setminus \{1\}$ such that $\PP$
	holds for all normal subgroups $N<H$ with $N\cap \FF=\emptyset$.
	We shall  need the following (see also \cite[Theorem 2.27]{dgo}):
	
	\begin{theorem}\cite[Theorem 2.5]{sun} \label{sun}
		Suppose that $G$ is a group with a hyperbolically embedded
		subgroup $H$.  Then $(G,H,N)$  has the Cohen-Lyndon property for all sufficiently deep $N < H$.
	\end{theorem}
	
	Combining Proposition \ref{prop-cl}, Theorems \ref{cw} and \ref{sun},
	and Remark \ref{rmk-ab}
	we immediately have the following:

	\begin{theorem}\label{thm-hypemb}
		Let $G$  have Property (T) and $H<G$ be hyperbolically embedded such that $H$ is one of the following:
		\begin{enumerate}
			\item the free group $F_n$ \emph{(}$n>1$\emph{)},
			\item the fundamental group of a 
			closed surface of genus greater than one,
			\item infinite abelian.
		\end{enumerate}
		Then, for all sufficiently deep finite index normal subgroups $N<H$, 
		the fiber product of two copies of the quotient map $G\to G/\llangle N\rrangle$ does not have Property (T).
	\end{theorem}
	
	We conclude this section with the following question:
	
	\begin{qn}\label{q-hypemb}
		In Theorem \ref{thm-hypemb}, if $H$ does not have Property (T), does the conclusion continue to hold? Shalom's generalization of the Delorme-Guichardet theorem in terms
		of reduced first cohomology \cite[\S 3.2]{BdlHV} might be relevant here.
	\end{qn}
		
	\section{Necessary and sufficient conditions}\label{sec-ns}
	In this section, we shall deal with the special case of a fiber product corresponding to two copies of the same surjective homomorphism $q:G\to H$.  Then, $G\times_HG\cong N\rtimes G$ by Lemma \ref{lem-gpth}, where $N$ is the kernel of $q$ and the action of $G$ on $N$ is via conjugation. For any semidirect product $N\rtimes G$ (not necessarily coming from a fiber product), we assume that $G$ has Property (T). We want to know when $N\rtimes G$ has Property (T). Property (T) can be expressed in terms of normalized functions of positive type \cite[Chapitre 5, Th\'eor\`eme 11]{dlHV} or functions conditionally of negative type \cite[Theorem 2.10.4]{BdlHV}. In our situation, we show that Property (T) can be expressed in terms of $G$ fixed points in the space of such functions on $N$. See Theorem \ref{fupoty} and Theorem \ref{funety} below. When $N$ is abelian we prove a similar result involving $G$ fixed points in the space of regular measures on the unitary dual of $N$, see Theorem \ref{prob}.

	\subsection{Normalized functions of positive type} 
	
	\begin{theorem}\label{fupoty}
		Let $G$ be a locally compact group having Property (T), which acts continuously by automorphisms on a $\sigma$-compact, locally compact group $N$. This induces an action of $G$ on the space $\mathcal{P}_1(N)$ of normalized functions of positive type on $N$. Let $\mathcal{P}_1(N)^G$ be the set of fixed points. Then $N\rtimes G$ has Property (T) if and only if every sequence $\{f_n\}\subset\mathcal{P}_1(N)^G$, that converges uniformly on compact subsets to the constant function $1$, converges uniformly on $N$.
	\end{theorem}
	\begin{remark}
		If we put $G=\{1\}$ in the above theorem then we get back the statement of \cite[Chapitre 5, Th\'eor\`eme 11]{dlHV}.
	\end{remark}
	The proof of
	\cite[Chapitre 5, Th\'eor\`eme 11]{dlHV} can be adapted to our situation using a couple of preliminary Lemmas. The first lemma can be seen as an adaptation of the GNS construction to our case. 
	\begin{lemma}\label{GNSfupoty}
		Let $G$ be a topological group acting continuously by automorphisms on a topological group $N$. Then any $f\in\mathcal{P}_1(N)^G$ is of the form $f(n)=\langle\Pi(n,1)v,v\rangle$, where $(\Pi,V_\Pi)$ is a unitary representation of $N\rtimes G$ and $v\in V_\Pi$ is a unit vector which is fixed by $\Pi(\{1\}\times G)$. 
	\end{lemma}
	
	\begin{proof}
		The $G$-action on $N$ induces a $G$-action on both sides of the GNS correspondence for the group $N$, as given in Theorem \ref{orgGNSfupoty}. Denoting the $G$ action on $N$ by $\phi$, the induced action on $\mathcal{P}_1(N)$ is given by $(g\cdot f)(x)=f(\phi(g)^{-1}x)$. The action on the equivalence class of cyclic representations is given by $(g\cdot\pi)(n)=\pi(\phi(g^{-1})n)$. Note that the map giving this correspondence is $G$-equivariant. Thus, if $f\in\mathcal{P}_1(N)^G$ then there exists $(\pi,V_\pi,v)$, whose equivalence class is $G$-invariant,  such that $f(n)=\langle\pi(n)v,v\rangle$. The $G$-invariance of the class of $(\pi,V_\pi,v)$ means that for each $g\in G$ there exists a unitary operator $T_g:V_\pi\to V_\pi$ such that $T_g(g\cdot\pi)(n)=\pi(n)T_g$, for all $n\in N$, and $T_gv=v$. Since $(\pi,V_\pi)$ is cyclic, therefore $T_g$ is the unique operator satisfying these properties. Uniqueness implies that $T_{g_1g_2}=T_{g_1}T_{g_2}$, for all $g_1,g_2\in G$. Now define a representation $(\Pi,V_\pi)$ of $N\rtimes G$ as $\Pi(n,g)=\pi(n)T_g$. Let us check this is indeed a representation: 
		\begin{align*}
			\Pi((n_1,g_1)(n_2,g_2)) &=\Pi(n_1\phi(g_1)(n_2),g_1g_2) =\pi(n_1\phi(g_1)n_2)T_{g_1g_2}\\
			&=\pi(n_1)\pi(\phi(g_1)n_2)T_{g_1}T_{g_2}=\pi(n_1)(g_1^{-1}\cdot\pi)(n_2)T_{g_1}T_{g_2}\\
			&=\pi(n_1)T_{g_1}\pi(n_2)T_{g_2}=\Pi(n_1,g_1)\Pi(n_2,g_2). 
		\end{align*}
		Now let us check that $\Pi$ is continuous. That is, we have to check that for any $w\in V_\pi$, the orbit map $G\to V_\pi, (n,g)\mapsto \Pi(n,g)w$ is continuous. Since $\Pi$ is unitary it is enough to check  continuity at $(1,1)$. Since $$\|\Pi(n,g)w-w\|=\|\Pi(n,1)\Pi(1,g)w-\Pi(1,g)w\|+\|\Pi(1,g)w-w\|$$ and $\Pi|_{N\times\{1\}}=\pi$ is already continuous,  it suffices to prove that $\Pi|_{\{1\}\times G}$ is continuous. Since the subspace $D$, which is the linear span of $\{\pi(n)v:n\in N\}$, is dense in $V_\pi$,  it is enough to show that $g\mapsto\Pi(1,g)w$ is continuous for all $w\in D$. Any $w\in D$ is of the form $\sum_{i=1}^k\lambda_i\pi(n_i)v$, where $\lambda_i\in\c$ and $n_i\in N$. Now the continuity of $\Pi|_{\{1\}\times G}$ follows from the continuity of $\phi,\pi$ and the following inequality:
		\begin{align*}
			&\|\Pi(1,g)(\sum_{i=1}^k\lambda_i\pi(n_i)v)-\sum_{i=1}^k\lambda_i\pi(n_i)v\|=\|\sum_{i=1}^k\lambda_iT_g\pi(n_i)v-\sum_{i=1}^k\lambda_i\pi(n_i)v\|\\
			=~&\|\sum_{i=1}^k\lambda_i\pi(\phi(g)n_i)T_gv-\sum_{i=1}^k\lambda_i\pi(n_i)v\|\le\sum_{i=1}^k|\lambda_i|\|\pi(\phi(g)n_i)v-\pi(n_i)v\|\qedhere
		\end{align*}
		\end{proof}
	The next lemma tells us that if $G$ has Property (T) then in a representation of $N\rtimes G$ which has almost invariant vectors, the almost invariant vectors can be chosen to be $G$-invariant. 
	
	\begin{lemma}\label{alinvect}
		Let $(\pi,V_\pi)$ be a unitary representation of $N\rtimes G$ which has almost invariant vectors. Let $V_\pi^G$ denote the subspace of $\pi(\{1\}\times G)$-invariant vectors. If $G$ has Property (T) then for any compact $K\subset N\rtimes G$ and $\delta >0$, there exists a unit vector $w\in V_\pi^G$ which is $(K,\delta)$-invariant.
	\end{lemma}
	
	\begin{proof}
		Without loss of generality we may assume that $K$ contains a compact generating set of $\{1\}\times G$ and $\delta\le 2$. By Proposition \ref{alinv}, there exists $\epsilon>0$ such that given a $(K,\epsilon)$-invariant unit vector $v$ there exists a unit vector $w\in V_\pi^G$ such that $\|v-w\|<\delta/3$. We may assume that $v$ is a $(K,\min(\epsilon,\delta/3))$-invariant unit vector. Then for all $k\in K$, we have
		\begin{align*}
			&\|\pi(k)w-w\|\le \|\pi(k)w-\pi(k)v\|+\|\pi(k)v-v\|+\|v-w\|\\
			\le~& \|\pi(k)\|\|v-w\|+\|\pi(k)v-v\|+\|v-w\| < \delta/3+\delta/3+\delta/3=\delta.\qedhere
		\end{align*}
	\end{proof}
	\noindent\textit{Proof of Theorem \ref{fupoty}.}
	($\Rightarrow$) Let $\{f_n\}$ be a sequence in $\mathcal{P}_1(N)^G$ which converges uniformly on compact subsets to $1$. Given $\delta >0$, we wish to show that for all large $n$, $|f_n(x)-1|<\delta$ for all $x\in N$. By Lemma \ref{GNSfupoty}, for each $n\in\n$, there exists a representation $(\pi_n,V_n)$ of $N\rtimes G$ and $v_n\in V_n$ such that $f_n(x)=\langle\pi_n(x,1)v_n,v_n\rangle$ and $\pi_n(\{1\}\times G)v_n=v_n$.  
	Since $N\rtimes G$ has Property (T) by assumption,  it has a compact generating set $K$. By Proposition \ref{alinv}, there exists $\epsilon>0$ such that if any representation $(\pi,V_\pi)$ of $N\rtimes G$ has a $(K,\epsilon)$-invariant unit vector $v$, then there exists a unit invariant vector $w$ such that $\|v-w\|<\delta/2$. We will apply this to the representations $\pi_n$. Let $p$ be the projection onto the first coordinate of $N\rtimes G$. By definition of convergence on compact subsets, for large $n$, $|1-f_n(x)|<\epsilon^2/2$ for all $x\in p(K)$. Then $$\|\pi_n(x,1)v_n-v_n\|^2\le 2|1-\langle\pi_n(x,1)v_n,v_n\rangle|<\epsilon^2,$$ for large $n$ and $x\in p(K)$. So for any $(x,g)\in K$ and large enough $n$, we have $$\|\pi_n(x,g)v_n-v_n\|=\|\pi_n(x,1)\pi_n(1,g)v_n-v_n\|=\|\pi_n(x,1)v_n-v_n\|<\epsilon.$$ Thus there exist unit invariant vectors $w_n\in V_n$ such that $\|v_n-w_n\|<\delta/2$. Now for large $n$ and all $x\in N$, we have
	\begin{align*}
		&|f_n(x)-1|=|\langle\pi_n(x,1)v_n,v_n\rangle-\langle\pi_n(x,1)w_n,w_n\rangle|\\
		=~&|\langle\pi_n(x,1)v_n,v_n\rangle-\langle\pi_n(x,1)v_n,w_n\rangle+\langle\pi_n(x,1)v_n,w_n\rangle-\langle\pi_n(x,1)w_n,w_n\rangle|\\
		\le ~& |\langle\pi_n(x,1)v_n,v_n-w_n\rangle|+|\langle\pi_n(x,1)(v_n-w_n),w_n\rangle|\\
		\le~ &\|\pi_n(x,1)\|\|v_n\|\|v_n-w_n\|+\|\pi_n(x,1)\|\|v_n-w_n\|\|w_n\|
		\le  2\|v_n-w_n\|<\delta.
	\end{align*}
	
	\noindent ($\Leftarrow$) Given a representation $\pi$ of $N\rtimes G$ which has almost invariant vectors, we will show that it must have an invariant vector. Since $N$ is $\sigma$-compact, there exists a sequence $K_1\subset K_2\subset\cdots$ of compact subsets of $N$, such that $\cup_iK_i=N$. By Lemma \ref{alinvect}, for each $n$, there exists $v_n\in V_\pi^G$ which is $(K_n,1/n)$-invariant. Define functions $f_n$ on $N$ as $$f_n(x):=\langle\pi(x,1)v_n,v_n\rangle.$$
	
	Observe first that the sequence $\{f_n\}$ is contained in $\mathcal{P}_1(N)^G$ and converges uniformly to $1$ on compact subsets of $N$. To see this, take any compact $K\subset N$, then $K\subset K_n$ for all large $n$ and hence for all $k\in K$, we have
	$$|f_n(k)-1|=|\langle\pi(k,1)v_n,v_n\rangle-1|\le\|\pi(k,1)v_n-v_n\|<1/n.$$ 
	
	By hypothesis, $f_n\to 1$ uniformly on $N$. Then for large $n$, we have $\sup_{x\in N}|f_n(x)-1|\le 1/2$. This implies that 
	$$\inf_{x\in N}\text{Re}\langle\pi(x,1)v_n,v_n\rangle\ge 1/2.$$ Then for any $(x,g)\in N\rtimes G$, we have $$\text{Re}\langle\pi(x,g)v_n,v_n\rangle=\text{Re}\langle\pi(x,1)\pi(1,g)v_n,v_n\rangle=\text{Re}\langle\pi(x,1)v_n,v_n\rangle\ge 1/2.$$ Finally Lemma \ref{bddorbitur} implies that $\pi$ has an invariant vector.\qed

	\subsection{Functions conditionally of negative type}
	The Delorme-Guichardet Theorem \cite[Theorem 2.12.4]{BdlHV} says that, for $\sigma$-compact locally compact groups, Property (T) is equivalent to Property (FH). Recall that a group $G$ is said to have Property (FH) if any affine isometric action of $G$ on a real Hilbert space has a fixed point. Property (FH) can be expressed in terms of functions conditionally of negative type \cite[Theorem 2.10.4]{BdlHV}. We now derive an analogous statement in the case of a semidirect product $N\rtimes G$, where $G$ already has Property (FH). 
	
	\begin{theorem}\label{funety}
		Let $G$ be a topological group having Property (FH), which acts continuously by automorphisms on a topological group $N$. This induces an action of $G$ on the space of functions conditionally of negative type on $N$. Then $N\rtimes G$ has Property (FH) if and only if every $G$-invariant function conditionally of negative type on $N$ is bounded.  
	\end{theorem}
	
	\begin{remark}
		Putting $G=\{e\}$ in the above theorem give us back the statement of \cite[Theorem 2.10.4, (i)$\Leftrightarrow$(iii)]{BdlHV}.
	\end{remark}
	As before, we prove a GNS construction lemma that will allow us to adapt the proof of \cite[Theorem 2.10.4]{BdlHV} to a proof of Theorem \ref{funety}. 
	
	\begin{lemma}\label{GNSfunety}
		Let $G$ be a topological group which acts continuously by automorphisms on a topological group $N$. Let $\psi$ be a $G$-invariant function conditionally of negative type on $N$. Then there exists an affine isometric action $(A,\mathcal{H})$ of $N\rtimes G$ such that $\psi(x)=\|A(x,1)(0)\|^2$.
		%and $A(1,g)(0)=0$ for all $g\in G$.  
	\end{lemma}
	
	The proof is very similar to that of Lemma \ref{GNSfupoty}, so we omit it.
	
	\vspace{.2cm}
	
	\noindent\textit{Proof of Theorem \ref{funety}.}
	($\Rightarrow$) Let $\psi$ be a $G$-invariant function conditionally of negative type on $N$. By Lemma \ref{GNSfunety}, there exists an affine isometric action $(\alpha,\mathcal{H})$ of $N\rtimes G$ such that $\psi(x)=\|\alpha(x,1)(0)\|^2$.
	%and $\alpha(1,g)(0)=0$ for all $g\in G$.
	Since $N\rtimes G$ has Property (FH) therefore $\alpha$ has a fixed point. By Lemma \ref{bddorbitai}, the corresponding $1$-cocycle $(x,g)\mapsto\alpha(x,g)(0)$ is bounded. Since $\psi$ is the restriction of the norm square of this map to $N\times\{1\}$, therefore $\psi$ is also bounded.
	
	\vspace{.1cm}
	
	\noindent($\Leftarrow$) Let $(\alpha,\mathcal{H})$ be an affine isometric action of $N\rtimes G$. We wish to show that $\alpha$ has a fixed point. By Lemma \ref{bddorbitai}, it is enough to show that the cocycle corresponding to $\alpha$ is bounded. That is, the map $$N\rtimes G\to\mathcal{H}, (x,g)\mapsto\alpha(x,g)(0)$$ is bounded.  The restriction of $\alpha$ to $\{1\}\times G$ is an affine isometric action of $G$. Since $G$ has Property (FH), there exists  $v\in\mathcal{H}$ such that $\alpha(1,g)v=v$, for all $g\in G$. The function $\psi:N\to\r$ given by $\psi(x)=\|\alpha(x,1)v-v\|^2$ is a function conditionally of negative type on $N$. The following calculation tells us that $\psi$ is $G$-invariant.
	\begin{align*}
		&\psi(\phi(g)x)=\|\alpha(\phi(g)x,1)v-v\|^2
		=\|\alpha((1,g)(x,1)(1,g^{-1}))v-v\|^2\\
		=~&\|\alpha(1,g)\alpha(x,1)v-v\|^2 =\|\alpha(1,g)\alpha(x,1)v-\alpha(1,g)v\|^2
		=\|\alpha(x,1)v-v\|^2.
	\end{align*}
	For any $(x,g)\in N\rtimes G$, $$\|\alpha(x,g)v-v\|^2=\|\alpha(x,1)\alpha(1,g)v-v\|^2=\|\alpha(x,1)v-v\|^2=\psi(x).$$ Thus, $$\|\alpha(x,g)(0)\|\le\|\alpha(x,g)(0)-\alpha(x,g)(v)\|+\|\alpha(x,g)v-v\|+\|v\|\le\sqrt{\psi(x)}+2\|v\|$$ and the last quantity is bounded by assumption. \qed
	
	\subsection{Abelian kernel: Invariant Probability Measures}
	
	Theorem \ref{fupoty} gives an equivalent condition for $N\rtimes G$ to have Property (T), when $G$ has Property (T), in terms of $G$-invariant functions of positive type on $N$. Due to availability of spectral theory, in the case $N$ is abelian, it is possible to give an equivalent condition in terms of the space $\mathcal{M}(\widehat{N})^G$ of regular $G$-invariant probability measures on the unitary dual $\widehat{N}$ of $N$. We put the subspace topology on $\mathcal{M}(\widehat{N})$ induced by the inclusion $\mathcal{M}(\widehat{N})\subset C_c(\widehat{N})^*$, where $C_c(\widehat{N})^*$ has the weak${}^*$ topology.
	
	\begin{theorem}\label{prob}
		Let $G$ be a locally compact group having Property (T), which acts continuously by automorphisms on a second countable locally compact abelian group $N$. Let $\delta_1\in\mathcal{M}(\widehat{N})$ be the Dirac mass at $1$. Then $N\rtimes G$ has Property (T) if and only if there is no sequence $\{\mu_n\}\subset\mathcal{M}(\widehat{N})^G$, such that $\mu_n\to\delta_1$ and $\mu_n(1)=0$, for all $n\in\n$.
	\end{theorem}
	
	This result is very similar to \cite[Theorem 5.1]{ioana} (for discrete groups) and \cite[Theorem 1, ($\neg$T)$\Leftrightarrow$(P)]{cortes}. Their result gives an equivalent condition for $(N\rtimes G,N)$ to have relative Property (T) in terms of sequences in $\mathcal{M}(\widehat{N})$. Since they do not assume that $G$ has Property (T), they put an appropriate condition on the sequences that reflects the existence of almost invariant vectors of $G$. Since we assume that $G$ has Property (T), considering $G$-invariant measures is enough for us. The remaining two conditions are the same. The proof of \cite[Theorem 1]{cortes} can be adapted to our case with the necessary modifications provided by Lemma \ref{alinvect}  and Lemma \ref{splrepns}. For the convenience of reader we give a complete proof below. 
	
	We begin with the analogs of Lemma \ref{GNSfupoty} and Lemma \ref{GNSfunety}.
	
	\begin{lemma}\label{splrepns}
		Let $G$ be a topological group that acts continuously by automorphisms on a locally compact abelian group $N$. Given $\mu\in\mathcal{M}(\widehat{N})^G$, there exists a representation $(\Pi,L^2_\mu(\widehat{N}))$ of $N\rtimes G$ such that $\Pi(\{1\}\times G)$ fixes the constant function $1$ in $L^2_\mu(\widehat{N})$.
	\end{lemma}
	
	\begin{proof}
		Consider the set of equivalence classes on the left hand side of the correspondence in Lemma \ref{directint}. In the proof of Lemma \ref{GNSfupoty} we  showed that if the class of $(\pi,V_\pi,v)$ is fixed by $G$, then $\pi$ can be extended to a representation $\Pi$ of $N\rtimes G$ in such a way that $v$ is fixed by $\Pi(\{1\}\times G)$. Thus it is enough to show that the bijective map between the two sets in Lemma \ref{directint} is $G$-equivariant. We need to check that if $\mu$ is the measure corresponding to  $(\pi,V_\pi,v)$ then $g\cdot\mu$, the pushforward of $\mu$ by $g$, is the measure corresponding to  $(g\cdot\pi,V_\pi,v)$. For this we note that the measure $\mu$ is in fact $P_{v,v}$, where $P$ is the $V_\pi$-projection valued measure associated to the representation $(\pi,V_\pi)$ \cite[Theorem 4.45]{folland}. The projection valued measure associated to $(g\cdot\pi,V_\pi)$ is $g\cdot P$. Hence $(g\cdot P)_{v,v}=g\cdot P_{v,v}=g\cdot\mu$. 
	\end{proof}
	
	Next we characterize the existence of a non-zero invariant vector in $L^2_\mu(\widehat{N})$ in terms of the measure $\mu$. This is the reason for putting the condition $\mu_n(\{1\})=0$ in Theorem \ref{prob}. The argument given here is a part of the proof of \cite[Theorem 7]{cortes}.
	
	\begin{lemma}\label{invect}
		Let $N$ be a second countable locally compact group
		and let $\mu$ be a regular probability measure on $\widehat{N}$. Let $(\rho,L^2_\mu(\widehat{N}))$ be the representation of $N$ given by $\rho(x)(f)(\chi):=\chi(x)f(\chi)$. Then $\rho$ has a non-zero invariant vector if and only if $\mu(\{1\})\ne 0$. 
	\end{lemma}
	\begin{proof}
		($\Leftarrow$) The characteristic function $\mathbb{I}_{\{1\}}$ is a non-zero invariant vector. 
		
		\noindent $(\Rightarrow)$ Suppose $\mu(\{1\})=0$. Let $f$ be an invariant vector. We wish to show that $f=0$ $\mu$-almost everywhere. That is, if $Z:=\{\chi\in\widehat{G}:f(\chi)\ne 0\}$, then we wish to show that $\mu(Z)=0$.  Invariance of $f$ means that for any $x\in N$, the function  $\chi\mapsto\chi(x)f(\chi)-f(\chi)$ vanishes $\mu$-almost everywhere. Hence there exists measurable $E_x\subset\widehat{G}$, with $\mu(E_x)=1$, such that $(\chi(x)-1)f(\chi)=0$ for all $\chi\in E_x$. Let $K_x$ be the closed set $\{\chi\in\widehat{G}:\chi(x)=1\}$. Then $\chi\in E_x\cap K^c_x$ implies $f(\chi)\ne 0$, that is, $\chi\in Z^c$. Hence $Z\subset E_x^c\cup K_x$ for all $x\in N$.  Since $\chi\ne 1$ means that there exists $x\in N$ such that $\chi(x)\ne 1$, we have $\{1\}^c=\cup_{1\ne x\in N}K_x^c$. Since $N$ is second countable, there exists a countable subset $\{x_i\}_i\subset N$ such that $\{1\}^c=\cup_iK_{x_i}^c$. Therefore by hypothesis, $\mu(\cap_i K_{x_i})=\mu(\{1\})=0$. Since  $$Z\subset\cap_i(E_{x_i}^c\cup K_{x_i})\subset\cap_i((\cup_iE_{x_i}^c)\cup K_{x_i})=(\cup_iE_{x_i}^c)\cup (\cap_iK_{x_i}),$$
		we have $$\mu(Z)\le\sum_i\mu(E_{x_i}^c)+\mu(\cap_iK_{x_i})=0,$$ that is, $\mu(Z)=0$.   
	\end{proof}
	
	The next result reformulates the convergence $\mu_n\to \delta_1$ in weak${}^*$ topology on $\mathcal{M}(\widehat{N})\subset C_c(\widehat{N})^*$ in terms of integration of the evaluation maps at each $x\in N$.
	
	\begin{lemma}\label{ev}\cite[Lemma 11]{cortes}
		If $N$ is $\sigma$-compact, then $\mu_n\to \delta_1$ in weak${}^*$ topology if and only if for all $x\in N$, $\int_{\widehat{N}}\chi(x)d\mu(\chi)\to 1$, uniformly on compact subsets of $N$.
	\end{lemma}
	
	Now we prove the main result of this subsection.
	
	\vspace{.2cm}
	
	\noindent\textit{Proof of Theorem \ref{prob}.}
	($\Rightarrow$) Let $\{\mu_n\}$ be a sequence in $\mathcal{M}(\widehat{N})^G$ such that $\mu_n\to\delta_1$ and $\mu_n(\{1\})=0$, for all $n\in\n$. We will show that $N\rtimes G$ does not have Property (T). For that we will produce a unitary representation $\Pi$ of $N\rtimes G$, which has almost invariant vectors but no invariant vector. By Lemma \ref{splrepns}, for each $n\in\n$, there exists a representation $(\Pi_n,L^2_{\mu_n}(\widehat{N}))$ of $N\rtimes G$  such that $\Pi_n(\{1\}\times G)1=1$. Let $\Pi:=\oplus_n\Pi_n$. By Lemma \ref{invect},  $\Pi_n$ does not have a non-zero $N$-invariant vector. Hence $\Pi$ does not have a non-zero $N$-invariant vector. In particular, $\Pi$ does not have a non-zero $N\rtimes G$ invariant vector. On the other hand we have
	\begin{equation*}
		\|\Pi_n(x,1)(1)-1\|^2=\int_{\widehat{N}}(\chi(x)-1)(\overline{\chi(x)}-1)d\mu_n(\chi)=2\text{Re}(1-\int_{\widehat{N}}\chi(x)d\mu_n(\chi)).
	\end{equation*}
	Since $\mu_n\to\delta_1$, Lemma \ref{ev} implies that the right hand side of the above equation converges to $0$ uniformly on compact subsets of $N$ as $n\to\infty$. Together with the fact that $1\in L^2_{\mu_n}(\widehat{G})$ is $\{1\}\times G$ invariant, this implies that $\Pi=\oplus_n\Pi_n$ has almost invariant vectors. 
	
	\noindent ($\Leftarrow$) Given a representation $\pi$ of $N\rtimes G$ which has almost invariant vectors, but no invariant vector, we will produce a sequence $\{\mu_n\}$ of regular $G$-invariant probability measures on $\widehat{N}$ such that $\mu_n\to\delta_1$ and $\mu_n(\{1\})=0$, for all $n\in\n$. Let $K_1\subset K_2\subset\cdots$ be a sequence of compact sets such that $\cup_iK_i=N$. By Lemma \ref{alinvect}, for each $n$, there exists a unit vector $v_n\in V_\pi^G$ which is $(K_n,1/n)$-invariant. Consider the cyclic subrepresentation $\pi_n$ of $\pi|_N$ generated by $v_n$. Via the correspondence defined in Lemma \ref{directint}, there exists a regular probability measure $\mu_n$ such that the pair $(\pi_n,v_n)$ is equivalent to the pair $(\rho_n,1)$, where $(\rho_n,L^2_{\mu_n}(\widehat{N}))$ is the representation given by $(\rho_n(x)f)(\chi):=\chi(x)f(\chi)$. Moreover $\pi(g)$ intertwines $g\cdot\pi_n$ and $\pi_n$, while sending $v_n$ to $v_n$. Thus the class of $(\pi_n,v_n)$ is a fixed point under $G$-action, hence $\mu_n\in\mathcal{M}(\widehat{N})^G$. Since $\pi$ has no invariant vector, $\pi_n$ has no invariant vector and hence $\rho_n$ has no invariant vector. Therefore, by Lemma \ref{invect}, $\mu_n(\{1\})=0$. On the other hand \begin{equation*}
		1/n^2>\|\pi(x)v_n-v_n\|^2=\|\rho_n(x)1-1\|^2=\int_{\widehat{N}}|\chi(x)-1|^2d\mu_n(\chi),
	\end{equation*}
	for all $x\in K_n$. This implies that $\int_{\widehat{N}}\chi(x)d\mu_n(\chi)\to 1$ uniformly on compact subsets of $N$. Hence, by Lemma \ref{ev}, $\mu_n\to\delta_1$ in the weak${}^*$ topology.   \qed

	\section{Finitely generated abelian and nilpotent kernels}\label{sec-special}
	
	We continue studying the question whether $N\rtimes G$ has Property (T) given $G$ has Property (T). We assume now that $N$ is a finitely generated abelian group. Let $N_f$ and $N_t$ denote the free and torsion parts of $N$ respectively. Since $G$ acts on $N_f$ and $N_t$ separately, we have $N\rtimes G\cong(N_f\times N_t)\rtimes G\cong(N_f\rtimes G)\times(N_f\rtimes G)$. Since $N_t$ is finite, $N_t\rtimes G$ has Property (T). Thus $N\rtimes G$ has Property (T) if and only if $N_f\rtimes G$ has Property (T).
	%Thus, when $N$ is finitely generated abelian, the question of $N\rtimes G$ having Property (T) depends only on the rank of $N$.
	If rank of $N$ is $n$, then $N_f\cong\z^n$ and $\text{Aut}(N_f)\cong\text{GL}(n,\z)$. With this identification, if $\phi$ denotes the action of $G$ on $N_f$, then $\phi(G)\subset\text{GL}(n,\z)$. 
	
	The next general lemma tells us that it is the image $\phi(G)$ of $G$ that matters in determining whether $N\rtimes G$ has Property (T). 
	
	\begin{lemma}\label{img}
		Let $G$ be a group which has Property (T) and $\phi:G\to\emph{Aut}(N)$ be an action of $G$ on $N$. Then $N\rtimes G$ has Property (T) if and only if $N\rtimes \phi(G)$ has Property (T).
	\end{lemma}
	
	\begin{proof}
		Since $N\rtimes\phi(G)$ is a quotient of $N\rtimes G$, therefore $N\rtimes G$ having Property (T) implies the same for $N\rtimes\phi(G)$. On the other hand suppose $N\rtimes\phi(G)$ has Property (T). Consider the short exact sequence
		\begin{equation*}
			1\to \{1\}\times\text{ker}(\phi)\to N\rtimes G\to N\rtimes\phi(G)\to 1.
		\end{equation*}
		Since $\{1\}\times\text{ker}(\phi)\subset\{1\}\times G\subset N\rtimes G$, with $G$ having Property (T), therefore $(N\rtimes G,\{1\}\times\text{ker}(\phi))$ has relative Property (T). Thus by Lemma \ref{extension}, $N\rtimes G$ has Property (T).
	\end{proof}
	\begin{cor}\label{imgnonT}
		With notation as in Lemma \ref{img}, if $\phi(G)$ is finite and $N$ does not have Property (T), then $N\rtimes G$ does not have Property (T).
	\end{cor}
	\begin{proof}
		By the above lemma, it is enough to prove that $\phi(G)\rtimes N$ does not have Property (T). If $\phi(G)$ is finite then $N$ is a finite index subgroup of $N\rtimes\phi(G)$. This implies that $N\rtimes\phi(G)$ has Property (T) if and only if $N$ has Property (T).  Since $N$ does not have Property (T), nor does $\phi(G)\rtimes N$.
	\end{proof}

	We need a fact that we now quote. It is a special cases of the following  result.
	% \begin{theorem}\label{haagerup}\cite[Theorem 2]{cor}
	% 	Any subgroup of $\emph{SL}(2,\z)$ has Haagerup property.
	% \end{theorem}  
	Let $H,N$ be subgroups of a group $G$. Then the triple $(G,H,N)$ is said to have relative Property (T) if any unitary representation of $G$, whose restriction to $H$ has almost invariant vectors, has an $N$ invariant vector. Note that if $(G,H,N)$ has relative Property (T), then so does the pair $(G,N)$.
	
	\begin{prop}\label{dichotomy}\cite[Corollary 3.1]{raja}
		Let $G$ be a subgroup of $\emph{GL}(n,\r)$. If $\r^n$ is $G$-irreducible then either $(\r^n\rtimes G,G,\r^n)$ has relative Property (T) or $G$ is contained in a compact extension of a diagonalizable group.
	\end{prop}
	%So from now on we may assume that $N$ is torsion free. Thus $N$ is isomorphic to $\z^n$, for some $n$. Note that $\text{Aut}(\z^n)=\text{GL}(n,\z)$.
	
	Now we state our main result in the case where $N$ is finitely generated abelian. 
	
	\begin{prop}\label{fingenab}
		Let $G$ be a group with Property (T) which acts by automorphisms on a finitely generated abelian group $N$ of rank $n$.
		%Let $f:G\to H$ be a surjective group homomorphism whose kernel $N$ is a finitely generated abelian subgroup of rank $n$.
		Let $\phi:G\to\emph{GL}(n,\z)$ be the action of $G$ on the free part of $N$.  
		\begin{enumerate}
			\item If $n=1$ or $2$, then $N\rtimes G$ does not have Property (T).
			\item If $n>2$ and $\phi(G)$ is Zariski dense in $\emph{SL}(n,\r)$, then $N\rtimes G$ has Property (T).
		\end{enumerate}
	\end{prop}
	
	\begin{proof}
		(1) By the discussion at the beginning of this section we may assume $N=\z$ or $\z^2$. If $N=\z$, then $\text{Aut}(N)=\text{GL}(1,\z)$ is finite and hence so is $\phi(G)$. Thus Corollary \ref{imgnonT} applies. If $N=\z^2$, then $\text{Aut}(N)=\text{GL}(2,\z)$. The group $\text{GL}(2,\z)$ has Haagerup property since the free group with two generators is a finite index subgroup of $\text{GL}(2,\z)$. Therefore  $\phi(G)$ also has the Haagerup property. On the other hand $\phi(G)$, being a quotient of a Property (T) group, has Property (T). Therefore $\phi(G)$ is finite. Again Corollary \ref{imgnonT} applies. 
		
		\vspace{.1cm}
		
		\noindent (2) By Lemma \ref{img}, it is enough to prove that $\z^n\rtimes\phi(G)$ has Property (T). Further, since $(\r^n\rtimes\phi(G))/(\z^n\rtimes\phi(G))\cong\r^n/\z^n$ is compact, it is enough to prove that $\r^n\rtimes\phi(G)$ has Property (T).  Since $\phi(G)$ is Zariski dense in $\text{SL}(n,\r)$ and $\r^n$ is $\text{SL}(n,\r)$-irreducible, therefore $\r^n$ is also $\phi(G)$-irreducible. Hence Proposition \ref{dichotomy} applies. If $\phi(G)$ is contained in a compact extension of a diagonalizable group, then it is amenable, since closed subgroups of amenable groups are amenable. On the other hand $\phi(G)$ has Property (T), hence $\phi(G)$ is finite. But finite subgroups are not Zariski dense. Hence $(\r^n\rtimes\phi(G),\r^n)$ has relative Property (T). Since $\phi(G)$ has Property (T), therefore by Lemma \ref{extension}, $\r^n\rtimes\phi(G)$ has Property (T).
	\end{proof}
	
	\begin{remark}
		In the proof of (2), we only used $n\ge 2$. But when $n=2$, the Haagerup property of $\text{GL}(2,\z)$ excludes the possibility of a homomorphism $\phi$ from a Property (T) group to $\text{GL}(2,\z)$ having Zariski dense image.	
	\end{remark}
	
	%\textcolor{blue}{State \cite[Theorem 1.2]{cms} below. Write a proof for proposition \ref{prop-nilp}.}
	Any result on relative Property (T) with abelian normal subgroup can be upgraded to one involving nilpotent normal subgroups via the following result. We denote the abelianization of $N$ by $N_{ab}$.
	
	\begin{theorem}\cite[Theorem 1.2]{cms}\label{thm-cms}
		Let $N$ be a closed nilpotent normal subgroup of a locally compact group $G$. Then $(G,N)$ has relative Property (T) if and only if $(G/[N,N],N_{ab})$ has relative Property (T). 
	\end{theorem}

	Thus we have the following result.
	
	\begin{prop}\label{prop-nilp}
		Let $G$ be a group with Property (T) which acts by automorphisms on a nilpotent group $N$. Suppose $N_{ab}$ is finitely generated with rank $n$. Let $\phi:G\to\emph{GL}(n,\z)$ be the induced action of $G$ on the free part of $N_{ab}$.  
		\begin{enumerate}
			\item If $n=1$ or $2$, then $N\rtimes G$ does not have Property (T).
			\item If $n>2$ and $\phi(G)$ is Zariski dense in $\emph{SL}(n,\r)$, then $N\rtimes G$ has Property (T).
		\end{enumerate}
	\end{prop}
	
	\begin{proof}
		Since $G$ has Property (T), by Lemma \ref{extension}, $(N\rtimes G, N)$ has relative Property (T) if and only if $N\rtimes G$ has Property (T). Similarly, $(N_{ab}\rtimes G, N_{ab})$ has relative Property (T) if and only if $N_{ab}\rtimes G$ has Property (T). Thus Theorem \ref{thm-cms} tells us that $N\rtimes G$ has Property (T) if and only if $N_{ab}\rtimes G$ has Property (T). The statement now follows  from Proposition \ref{fingenab}.
	\end{proof}

	\section*{Acknowledgments} We thank Francois Dahmani for explaining to us a simple geometric proof of Theorem \ref{sun}. We also thank Indira Chatterji for
	helpful comments on an earlier draft. In particular, she asked us a version of Question \ref{q-hypemb} when $H$ is assumed to have the Haagerup property. Special thanks are due to the referee for a careful reading and several very helpful comments.

\end{document}